\documentclass[12pt,twoside]{article}
\usepackage[left=2cm,right=2cm,top=2cm,bottom=3cm]{geometry}
\RequirePackage{etex}
\usepackage[utf8]{inputenc}
\usepackage{amsmath,amssymb,latexsym}
\usepackage{authblk} 

\usepackage[giveninits=true]{biblatex} 
\usepackage{csquotes} 

\usepackage{mathrsfs}
\usepackage{graphicx}
\usepackage{float}
\usepackage{titling}
\usepackage[boxed,commentsnumbered,linesnumbered,ruled]{algorithm2e}
\SetKwInOut{Input}{Input}
\SetKwInOut{Output}{Output}
\SetKwInOut{Initialize}{Initialize}

\usepackage{color}
\usepackage{amsmath,amssymb,latexsym}
\usepackage{amsthm}
\usepackage[english]{babel}
\usepackage{multicol}
\usepackage{tikz}
\usetikzlibrary{shapes,fit}
\usetikzlibrary{decorations.pathreplacing, positioning, calc}

\usepackage{enumitem}

\usepackage{hyperref}
\hypersetup{colorlinks,allcolors=black}

\usepackage{subcaption}
\usepackage{pgf,tikz}
\usetikzlibrary{positioning}
\usepackage{wrapfig}
\usepackage{graphicx}
\usepackage{caption}
\usepackage{float}
\usepackage{xcolor}
\usepackage[flushleft]{threeparttable} 
\graphicspath{{pictures/}}

\newtheorem{theorem}{Theorem}
\newtheorem{definition}{Definition}
\newtheorem{example}{Example}
\newtheorem{proposition}{Proposition}

\newtheorem{remark}{\textit{Remark}}

\makeatletter
\def\moverlay{\mathpalette\mov@rlay}
\def\mov@rlay#1#2{\leavevmode\vtop{%
   \baselineskip\z@skip \lineskiplimit-\maxdimen
   \ialign{\hfil$\m@th#1##$\hfil\cr#2\crcr}}}
\newcommand{\charfusion}[3][\mathord]{
    #1{\ifx#1\mathop\vphantom{#2}\fi
        \mathpalette\mov@rlay{#2\cr#3}
      }
    \ifx#1\mathop\expandafter\displaylimits\fi}
\makeatother

\newcommand{\bigcupdot}{\charfusion[\mathop]{\bigcup}{\cdot}}

\begin{document}
    \title{\vspace{-2cm}On the Activities and Partitions of the Vertex Subsets of Graphs}
\author{Kristina Dedndreaj\thanks{This research has been funded by the European Social Fund (ESF).} }
\author{ Peter Tittmann}
\affil{University of Applied Sciences Mittweida}

\date{\today}
\maketitle

\begin{abstract}
    Crapo introduced a construction of interval partitions of Boolean lattice for sets equipped with matroid structure. This construction, in the context of graphic matroids, is related to the notion of edge activities introduced by Tutte. This implies that each spanning subgraph of a connected graph can be constructed from edges of exactly one spanning tree by deleting a unique subset of internally active edges and adding a unique subset of externally active edges. Since the family of  vertex independent sets does not give rise to a matroid structure we can not apply Crapo's construction on the vertex set when using the family of independent sets as generating sets. In this paper, we introduce the concept of vertex activities to tackle the problem of generating interval partitions of the Boolean lattice of the vertex set. We show how to generate a cover, present some properties related to vertex activities of some special maximal independent sets and consider some special graphs. Finally, we will show that level labellings in pruned graphs always generate a partition.
\end{abstract}

\section{Introduction}
 \subsection{Definitions and Notions}
Here we will mention briefly the main concepts which are used in this paper. The  terminology which is not specified here conforms to that of standard textbooks in Graph Theory.

We will deal with simple, connected and undirected graphs $G=(V(G),E(G))=(V,E)$ where $V=\{1,\ldots,n\}$. A \textit{vertex labelling} of $G$ is a bijective mapping from $V$ to the set $\{1,\ldots,n\}$. A graph $H$ is a \textit{subgraph} of $G$ if $V(H)\subseteq V(G)$ and $E(H)\subseteq E(G)$. A subgraph $H$ of $G$ is spanning if $V(H)=V(G)$. If $H$ is a subgraph of $G$, then $G$ is said to be a supergraph of $H$. If $S\subseteq V(G)$, then the \textit{induced subgraph} $G[S]$ is the graph whose vertex set is $S$ and its edge set consists of edges of $E(G)$ which have both endpoints in $S$. If $S=\{v\}$ for some $v\in V$, then we write $G[v]$. 

A subset of vertices of $G$ is called \textit{independent} if no two vertices are adjacent. A \textit{maximal independent set} is an independent set which is not a proper  subset of any independent set. The family of maximal independent sets of a graph $G$ is denoted by $\mathscr{M}(G)$.

Let $v\in V$, the \textit{open neighbourhood} of the vertex $v$, denoted by $N(v)$ is the set of vertices which are adjacent to $v$ in $G$,

\begin{equation*}
    N(v)=\{u\in V: \{u,v\}\in E\}.
\end{equation*}
The \textit{closed neighbourhood} of $v$, denoted $N[v]$ is the set of vertices adjacent to $v$ together with $v$. That is, $N[v]=N(v)\cup \{v\}$. Analogously, for $S\subseteq V$ we define 
    \begin{equation*}
        N(S)=\bigcup\limits_{v\in S}N(v)
    \end{equation*}
and $N[S]=N(S)\cup S$. A set $S\subseteq V$ is called \textit{dominating} if $N[S]=V$.
If $A\subseteq B\subseteq V$, the set $[A;B]=\{X:A\subseteq X \subseteq B\}$ is called an \textit{interval}. A \textit{cover} of $V$ is a set of non-empty intervals $[A_1;B_1],\ldots,[A_k;B_k]$ such that 
    \begin{equation*}
        V=\bigcup\limits_{i=1}^k [A_i;B_i].
    \end{equation*}
If the intervals $[A_1;B_1],\ldots,[A_k;B_k]$ are pairwise disjoint they form a \textit{partition} or \textit{interval partition} of $V$.

\subsection{Background and Motivation}
Crapo introduced in \cite{crapo1969tutte} a construction of the interval partitions of the Boolean lattice of a set equipped with matroid structure.  This construction is related to the notion of activities as defined by Tutte in \cite{tutte1954contribution} which he used to state the polynomial which is now called after his name. The Tutte polynomial of a graph $G$ is denoted by $T(G;x,y)$ and defined as 
    \begin{equation}
         T(G;x,y)=\sum\limits_{\text{T  is a spanning tree of G}}x^{|i(T)|}y^{|e(T)|},
    \end{equation}
where  $i(T)$ is the set of internal active edges of the tree $T$ and $e(T)$ is the set of external active edges of the tree $T$. This notion of activities assumes a linear order on the edge set. As shown in \cite{tutte1954contribution}, the Tutte polynomial is independent of the linear order defined on the edge set.

Since the family of trees of a simple, connected, undirected graph gives rise to a matroid structure on the edge set, we can use this result to construct the family of spanning subgraphs of such graphs. This means that each spanning subgraph of a connected graph can be constructed from edges of exactly one spanning spanning tree by deleting a unique subset of
internally active edges and adding a unique subset of externally active edges. That is,
    \begin{equation}
        2^E=\bigcupdot_{\text{T is a spanning tree of }G}[T\setminus i(T); T\cup e(T)].
    \end{equation}
  For background on the Tutte polynomial the reader is referred to \cite{brylawski1992tutte}, \cite{ellis2011graph} and  \cite{sokal2005multivariate}.
  Different generalizations of the Tutte activities have been introduced in the literature, for example \cite{las1984tutte}, \cite{gordon1990generalized}, \cite{gessel1996tutte} and \cite{bernardi2008characterization} which also induce interval partition of subgraphs.
 
   The family of vertex independent sets does not give rise to a matroid structure in the vertex set because it does not satisfy the augmentation axiom. Therefore, we can not rely on the result of Crapo \cite{crapo1969tutte} to generate interval partitions of the Boolean lattice using the family of vertex independent sets as generating sets. In this paper we will introduce the  notion of vertex  activities with the intention of using it to generate interval partitions of the Boolean lattice of the vertex set.
  
\section{Vertex Activities}
\begin{definition}\label{extactivity}
  Let $G$ be a graph and let $A$ be an independent set of $G$. A vertex $v\in V\setminus A$ is called externally active in $A$ if there exists $a\in A$ such that $v\in N(a)$ and $v>a$. We denote the set of all externally active vertices with respect to $A$ by $Ext(A)$. So we have 
   $$Ext(A)=\{v\in A^{\complement}: \exists a\in A:a\in N(v) \text{ with } v>a\}.$$
\end{definition}

That is, a vertex outside an independent set $A$ is externally active with respect to $A$ if it is adjacent to and greater than some vertex in $A$. For a vertex $v$ in an independent set $A$ we define 

\begin{equation}
Subs(v)=\{u \in N(v): \big( A\setminus \{v\} \big)\cup \{u\} \text{ is independent}\},
\end{equation}

in other words, $Subs(v)$ is the subset of neighbors of $v$ which can substitute $v$ in $A$ while maintaining independence.

\begin{definition}\label{intactivity}
Let $G$ be a graph and let $A$ be an independent set of $G$. A vertex $v\in A$ is called internally active in $A$ if 
\begin{enumerate}
    \item $Subs(v)=\emptyset$ or
    \item $v>\max\{Subs(v)\}$ when $Subs(v)\neq\emptyset$.
\end{enumerate}

\end{definition}

That is, a vertex $v$ is internally active with respect to some independent set $A$ if $v$ is irreplaceable by any of its neighbors with greater label. We denote by $Int(A)$ the set of all internally active vertices of an independent set $A$ of $G$.

\begin{example}
In the graph on Figure \ref{fig1}, $Ext(\{3,5\})=\{4\}$ and $Int(\{3,5\})=\{5\}$.
\end{example}
 \begin{figure}[H]
 \centering
 \begin{tikzpicture}[
V/.style = {
            circle,thick,fill=white},scale=0.7]
\node[V] (B) at (-1,0) [scale=0.6,draw,fill=black!15] {3};
\node[V] (C) at (1,0) [scale=0.6,draw] {4};
\node[V] (D) at (-2,-2) [scale=0.6,draw] {2};
\node[V] (1) at (2,-2) [scale=0.6,draw] {1};
\node[V] (2) at (0,-3.5) [scale=0.6,draw,fill=black!15] {5};
\draw[black,very thick]
(B) to (C)
(B) to (D)
(D) to (C)
(D) to (2)
(2) to (1);
\end{tikzpicture}

\caption{A graph on 5 vertices}
\label{fig1}
\end{figure}
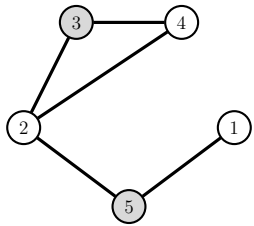

For an independent set $A$ we call $[A\setminus Int(A);A\cup Ext(A)]$ the interval generated by $A$. It turns out that the activities according to Definitions \ref{extactivity} and \ref{intactivity} always generate a cover. The proof of the following theorem is inspired by the proof of the Theorem 2.1 in \cite{AMC229}.

\begin{theorem}\label{covertheorem}
Let $G$ be a graph  with $V(G)=\{1,\ldots,n\}$ and $\mathscr{M}$ the family of maximal independent sets of $G$. The activities induce a cover of the Boolean lattice of the vertex set, i.e.,
        \begin{equation}\label{cover}
            \bigcup_{A \in \mathscr{M}}[A\setminus Int(A); A\cup Ext(A)]=2^V.
        \end{equation}
\end{theorem}
\begin{proof}
     Since $[A\setminus Int(A); A\cup Ext(A)]\subseteq V$ for all $A\in \mathscr{M}$ it is clear that 
     $$\bigcup_{A \in \mathscr{M}}[A\setminus Int(A); A\cup Ext(A)]\subseteq 2^V.$$
     Now we prove the converse relation. Let $A\subseteq V$, we will construct a maximal independent set $B$ such that 
     $$A\in [B\setminus Int(B); B\cup Ext(B)].$$
     We arrange the vertices $V$ in a sequence $v_1,v_2,\ldots,v_n$ such that the vertices of $A$ come first in increasing order and the vertices of $V\setminus A$ come after in decreasing order.
     
     \begin{equation*}
    \underbrace{\overrightarrow{v_1,v_2,\ldots ,v_{|A|}}}_{A}, \underbrace{\overleftarrow{v_{|A|+1},\ldots,v_{n}}}_{V\setminus A}
    \end{equation*}
     
     We start with the empty set and add the vertices of $V$ in the order they appear in the sequence  if the set remains independent. We set $B^{(0)}=\emptyset$ and for $i=1,\ldots,n$ 
     
     $$B^{(i)}=\begin{cases}
     B^{(i-1)}\cup \{v_i\} &\text{if } B^{(i-1)}\cup \{v_i\} \text{ is independent}\\
     B^{(i-1)} &\text{otherwise}.
     \end{cases}$$
     
     $B^{(n)}$ is an independent dominating set and thus a maximal independent set. Now we put $B=B^{(n)}$.
    Let $v$ be a vertex in $ A\setminus B$. This vertex is not added to $B$ because it is adjacent to one of the vertices added earlier with lower index. Therefore $v$ is an external active vertex of $B$, that is
     
     \begin{equation}\label{eq1}
         A\setminus B=A_e\subseteq Ext(B).
     \end{equation}

    A vertex $v\in B\setminus A$ is added to some $B^{(i)}$, where $|A|\leq i\leq n$, because $B^{(i)}\cup \{v\}$ is independent and $v$ is the first and thus the greatest vertex of $V\setminus A$ which can be added to $B^{(i)}$ while maintaining independence. Therefore the only vertices (if any) which are in $N(v)$  and that are greater than  $v$ are to be searched in $A\setminus B$. Since these vertices cannot substitute $v$ because they are adjacent to some vertices added earlier in $B$ then $v$ is an internally active vertex of $B$, that is
    
     \begin{equation}\label{eq2}
         B\setminus A=A_i\subseteq Int(B).
     \end{equation}
     
     Combining \eqref{eq1} and \eqref{eq2} we have 
     
      \begin{equation}\label{eq3}
         A=(B\setminus A_i)\cup A_e \in [B\setminus Int(B); B\cup Ext(B)].
     \end{equation}
     
\end{proof}

\begin{remark}
    Theorem \ref{covertheorem} works for any linear order on the vertex set.
\end{remark}

The intervals of the cover constructed according to \eqref{cover} depend on the labelling of the vertices $V$ and it turns out that for some labellings some subsets may be generated by more than a maximal independent set. 

\begin{example}
\begin{figure}[H]
\begin{minipage}{.3\textwidth}
\centering
\begin{tikzpicture}[
V/.style = {
            circle,thick,fill=white},scale=0.7]
\node[V] (5) at (0,0) [scale=1,draw] {};
\node[V] (4) at (2,0) [scale=1,draw] {};
\node[V] (1) at (2,2) [scale=1,draw] {};
\node[V] (2) at (0,2) [scale=1,draw] {};
\node[V] (3) at (4,1) [scale=1,draw] {};
\draw[black](2) to node {} (1)
(2) to node {} (5)
(2) to node {} (4)
(1) to node {} (5)
(1) to node {} (4)
(1) to node {} (3)
(5) to node {} (4)
(4) to node {} (3);
\end{tikzpicture}
\caption{Graph G}
\label{fig:sub0}
\end{minipage}
 ~
\begin{minipage}{.3\textwidth}
\centering
\begin{tikzpicture}[
V/.style = {
            circle,thick,fill=white},scale=0.7]
\node[V] (5) at (0,0) [scale=0.6,draw] {5};
\node[V] (4) at (2,0) [scale=0.6,draw] {4};
\node[V] (1) at (2,2) [scale=0.6,draw] {1};
\node[V] (2) at (0,2) [scale=0.6,draw] {2};
\node[V] (3) at (4,1) [scale=0.6,draw] {3};
\draw[black](2) to node {} (1)
(2) to node {} (5)
(2) to node {} (4)
(1) to node {} (5)
(1) to node {} (4)
(1) to node {} (3)
(5) to node {} (4)
(4) to node {} (3);
\end{tikzpicture}
\caption{Labelling 1}
\label{fig:sub1}
\end{minipage}
 ~
\begin{minipage}{.3\textwidth}
\centering
\begin{tikzpicture}[
V/.style = {
            circle,thick,fill=white},scale=0.7]
\node[V] (5) at (0,0) [scale=0.6,draw] {5};
\node[V] (4) at (2,0) [scale=0.6,draw] {1};
\node[V] (1) at (2,2) [scale=0.6,draw] {2};
\node[V] (2) at (0,2) [scale=0.6,draw] {4};
\node[V] (3) at (4,1) [scale=0.6,draw] {3};
\draw[black](2) to node {} (1)
(2) to node {} (5)
(2) to node {} (4)
(1) to node {} (5)
(1) to node {} (4)
(1) to node {} (3)
(5) to node {} (4)
(4) to node {} (3);
\end{tikzpicture}
\caption{Labelling 2}
\label{fig:sub2}
\end{minipage}
\end{figure}

We consider graph $G$ in Figure \ref{fig:sub0} and two different labellings of its vertices as shown in Figure \ref{fig:sub1} and Figure \ref{fig:sub2}. For notation convenience we will write, for example $[1;12345]$ instead of $[\{1\};\{1,2,3,4,5\}]$.

\begin{table}[H]
\begin{threeparttable}
\centering
\begin{tabular}{|p{3.65cm}|p{3.55cm}|p{3.55cm}|p{3.35cm}|}
\hline
Max. ind. set&  Int. active vertices& Ext. active vertices & Generated interval\\
 \hline
 \{1\}&  $\;\emptyset$& \{2,3,4,5\}& [1; 12345]\\
 \{2,3\}&  \{3\}& \{4,5\}& [2; 2345]\\
 \{3,5\}& \{3,5\}&  $\;\emptyset$& [$\emptyset$; 345]\\
 \{4\}& $\;\emptyset$&  \{5\}&  [4; 45]\\
\hline
\end{tabular}
\begin{tablenotes}
    \item
      \begin{flushright}
      2 repeated sets \; \;\;
      \end{flushright}
    \end{tablenotes}
\caption{Cover for labelling 1}
\label{table1}
\end{threeparttable}
\end{table}

\begin{table}[H]
\begin{threeparttable}
\centering
    \begin{tabular}{ |p{3.65cm}|p{3.55cm}|p{3.55cm}|p{3.5cm}| }
    \hline
    Max. ind. set&  Int. active vertices&   Ext. active vertices& Generated interval\\
     \hline
     \{1\}&  $\;\emptyset$& \{2,3,4,5\}& [1; 12345]\\
     \{2\}&  $\;\emptyset$& \{3,4,5\}& [2; 2345]\\
     \{3,4\}& \{3\}&  \{5\}& [4; 345]\\
     \{3,5\}& \{3,5\}&  $\;\emptyset$&  [$\emptyset$; 35]\\
    \hline
    \end{tabular}
    \begin{tablenotes}
    \item
      \begin{flushright}
      0 repeated sets \; \;\;
      \end{flushright}
    \end{tablenotes}
\caption{Cover for labelling 2}
\label{table2}
\end{threeparttable}
\end{table}

In the cover constructed from the Labelling 1, the sets $\{4\}$ and $\{4,5\}$ are generated by maximal independent set $\{3,5\}$ and maximal independent set $\{4\}$. This means that the cover constructed from Labelling 1 is not an interval partition. In the cover constructed form the Labelling 2, all the sets are generated by exactly one maximal independent set, i.e., they belong to exactly one interval. This means that the cover constructed from Labelling 2 is an interval partition.
\end{example}

\begin{definition}\label{extcomplete}
Let $G$ be a graph with vertex set $V=\{1,2\ldots,n\}$. A maximal idependent set $S$ of $G$ is called externally complete if $Ext(S)=S^{\complement}$.
\end{definition}

\begin{theorem}\label{extcompleteth}
Let $G$ be a graph with vertices $V=\{1,\ldots,n\}$. Then $G$ has a unique externally complete set $S$.
\end{theorem}

\begin{proof}
We will construct the required set $S$ by using Algorithm \ref{alg2}.

\vspace{0.25cm}
\begin{algorithm}[H]
\Input{A graph $G$ with vertex set $V=\{1,\ldots,n\}$}
\Output{An externally complete set $S$ of $G$}
\Initialize{$S:=\emptyset$ and $A:=V$}
 \While{While $A\neq \emptyset$}{
         Select the vertex $v\in A$ with the smallest label and add it to $S$\;
         Remove from $A$ the vertex $v$ and its neigbours\;
         }
 \caption{Construction of an externally complete set}
 \label{alg2}
\end{algorithm}

\vspace{0.25cm}
According to Definition \ref{extcomplete}, we need to prove that the output set $S$ of Algorithm \ref{alg2} is a maximal independent set and that it satisfies $Ext(S)=S^{\complement}$. By the structure of the Algorithm \ref{extcomplete} it is evident that $S$ is a maximal independent set. By definition of external activity it is clear that $Ext(S)\subseteq S^{\complement}$. Now we prove that $S^{\complement}\subseteq Ext(S)$. If some vertex $a$ is not in $S$ it means that it is adjacent to some vertex $b\in S$ with smaller label, thus $a \in Ext(S)$. To show uniqueness we assume the opposite, that the graph $G$ has two different externally complete sets $A$ and $B$. Let
    \begin{align*}
      a&=\min \{A\setminus B\}\\
      b&=\min \{B\setminus A\}.
    \end{align*}
Without loss of generality assume that $a<b$. By external completeness of $B$ we have 
    \begin{equation}
       \forall u\in B^{\complement}, \exists v\in B: u\in N(v) \text{ and } u>v.
    \end{equation}
This means that if we put $u=a$ we should find $v\in B$ such that $a\in N(v)$ and $a>v$. Since $a\in A \setminus B$, $a\in N(v)$ and $A$ is independent, then $v$ cannot be in $A$. Thus we have $v\in B\setminus A$ and  $a> v \geq b$, which is a contradiction.
\end{proof}

\begin{definition}\label{intcomplete}
Let $G$ be a graph with vertex set $V=\{1,2\ldots,n\}$. A maximal independent set $S$ of $G$ is called internally complete if $Int(S)=S$.
\end{definition}

 
\begin{theorem}\label{intcompletecons}
Let $G$ be a graph with vertices $V=\{1,\ldots,n\}$. Then $G$ contains an internally complete set $S$.
\end{theorem}
\begin{proof}
     We will construct the required set $S$ by using Algorithm \ref{Alg1}.
\vspace{0.25cm}

\begin{algorithm}[H]
\Input{A graph $G$ with vertex set $V=\{1,\ldots,n\}$}
\Output{An internally complete set $S$ of $G$}
\Initialize{$S:=\emptyset$ and $A:=V$}
 \While{While $A\neq \emptyset$}{
         Select the vertex $v\in A$ with the greatest label and add it to $S$\;
        Remove from $A$ the vertex $v$ and its neighbours\;
        }
 \caption{Construction of an internally complete set}
 \label{Alg1}
\end{algorithm}

\vspace{0.5cm}
According to Definition \ref{intcomplete} we need to prove that the output set $S$ of Algorithm \ref{Alg1} satisfies $Int(S)=S$ and that it is a maximal independent set. By the structure of the Algorithm \ref{Alg1} it is evident that $S$ is a maximal independent set. By the definition of internal activity it is clear that $Int(S)\subseteq S$. Now we prove  the converse. Let $a$ be an element of $S$. 
If $Subs(a)=\emptyset$, then $a$ is internally active. Now we consider the case when $Subs(a)\neq \emptyset$. Let $b\in Subs(a)$ such that $a<b$. Then $b\in N(c)$ for some $c$ which has been added to $S$ in some preceding step, otherwise the algorithm would have chosen $b$ instead of $a$. As a result the set 
       \begin{equation*}
            (S\setminus \{a\}) \cup \{b\}
        \end{equation*}
is not independent since $b$ is adjacent to $c\in S\setminus \{a\}$, which is a contradiction.
\end{proof}

\begin{remark}
    Internally complete sets are not unique. There can be constructed small graphs which have more than one internally complete set. 
\end{remark}

\begin{example}

\begin{figure}[H]
    \begin{minipage}[t]{.3\textwidth}
    \centering
    \begin{tikzpicture}[
    V/.style = {
                circle,thick,fill=white},scale=0.355]
    \node[V] (1) at (3,3) [scale=0.5,draw] {1};
    \node[V] (2) at (5,1) [scale=0.5,draw] {2};
    \node[V] (3) at (5,5) [scale=0.5,draw] {3};
    \node[V] (4) at (1,5) [scale=0.5,draw] {4};
    \node[V] (5) at (1,1) [scale=0.5,draw] {5};
    \draw[black]
    (1) to (2)
    (1) to (3)
    (1) to (4)
    (1) to (5)
    (2) to (3)
    (3) to (4)
    (4) to (5)
    (5) to (2);
    \end{tikzpicture}
    \caption{}
    \label{fig:sub3}
    \end{minipage}
     ~
    \begin{minipage}[t]{.3\textwidth}
    \centering
    \begin{tikzpicture}[
    V/.style = {
                circle,thick,fill=white},scale=0.65]
    \node[V] (3) at (0,0) [scale=0.6,draw] {3};
    \node[V] (5) at (2,0) [scale=0.6,draw] {5};
    \node[V] (1) at (2,2) [scale=0.6,draw] {1};
    \node[V] (4) at (0,2) [scale=0.6,draw] {4};
    \node[V] (2) at (4,1) [scale=0.6,draw] {2};
    \draw[black](2) to node {} (1)
    (2) to node {} (5)
    (3) to node {} (5)
    (1) to node {} (5)
    (1) to node {} (4)
    (1) to node {} (3)
    (5) to node {} (4)
    (4) to node {} (3);
    \end{tikzpicture}
    \caption{}
    \label{fig:sub4}
    \end{minipage}
     ~
    \begin{minipage}[t]{.3\textwidth}
    \centering
    \begin{tikzpicture}[
    V/.style = {
                circle,thick,fill=white},scale=0.55]
    \node[V] (3) at (-1,0) [scale=0.5,draw] {3};
    \node[V] (4) at (1,0) [scale=0.5,draw] {4};
    \node[V] (5) at (-2,-2) [scale=0.5,draw] {5};
    \node[V] (1) at (2,-2) [scale=0.5,draw] {1};
    \node[V] (2) at (0,-3.5) [scale=0.5,draw] {2};
    \draw[black]
    (4) to (3)
    (4) to (1)
    (4) to (5)
    (5) to (2)
    (5) to (3)
    (3) to (1)
    (2) to (1);
    \end{tikzpicture}
    \caption{}
    \label{fig:sub5}
    \end{minipage}
    \end{figure}
    
The graph in Figure \ref{fig:sub3} has the  internally complete sets $\{3,5\}$ and $\{2,4\}$. The graph in Figure \ref{fig:sub4} has the internally complete sets $\{5\}$ and $\{2,4\}$. The graph in Figure \ref{fig:sub5} has the internally complete sets $\{1,5\}$ and $\{2,4\}$.
\end{example}

\begin{definition}
    A maximal independent set $A$ is called \textit{complete} if it is both internally complete and externally complete.
\end{definition}

 \begin{theorem}\label{compl}
 If a graph $G=(V,E)$ has a complete maximal independent set $A$ then it is unique. Moreover, it generates all the subsets of $V$.
 \end{theorem}
 \begin{proof}
    Assume the opposite, that the graph $G$ has another complete maximal independent set $B$. Then $B$ has to be externally complete. As a consequence, we have two different maximal independent sets $A,B$ which are externally complete, a contradiction to Theorem \ref{extcompleteth}. The complete maximal independent set $A$ generates the interval
    \begin{equation}
        [A\setminus Int(A);A\cup Ext(A)]=[\emptyset; V]=2^V.
    \end{equation}
 \end{proof}

 \begin{remark}
    Since the complete maximal independent set  of a graph $G$  generates all the subsets of its vertex set, it is of interest to investigate labelling strategies of the vertices of $G$ which give rise to such sets.
\end{remark}

\begin{example}
The graph in Figure \ref{fig:fig14} has the complete maximal independent set $\{1,4,5,6,8,10\}$ and the graph in Figure \ref{fig:fig15} has the complete maximal independent set $\{1,2,3,7,8,9\}$.
\begin{figure}[H]
\centering
\begin{minipage}[t]{.45\textwidth}
        \centering
 \begin{tikzpicture}[
        V/.style = {
                    circle,thick,fill=white},scale=0.3]
    
        \node[V] (1) at (7,5) [scale=0.6,draw,fill=black!15] {1};
        \node[V] (2) at (12,6) [scale=0.6,draw] {2};
        \node[V] (3) at (9,7) [scale=0.6,draw] {3};
        \node[V] (4) at (9,9) [scale=0.6,draw,fill=black!15] {4};
        \node[V] (5) at (12,3) [scale=0.6,draw,fill=black!15] {5};
        \node[V] (6) at (7,1) [scale=0.6,draw,fill=black!15] {6};
        \node[V] (7) at (4,3) [scale=0.6,draw] {7};
        \node[V] (8) at (13,8) [scale=0.6,draw,fill=black!15] {8};
        \node[V] (9) at (10,1) [scale=0.6,draw] {9};
        \node[V] (10) at (1,3,) [scale=0.6,draw,fill=black!15] {10};

        \draw[black,thick]
        (1) to (2)
        (1) to (3)
        (1) to (7)
        (2) to (5)
        (2) to (8)
        (3) to (4)
        (3) to (5)
        (3) to (8)
        (5) to (9)
        (6) to (7)
        (6) to (9)
        (7) to (10);
    \end{tikzpicture}
\caption{}
\label{fig:fig14}
\end{minipage}
\begin{minipage}[t]{.45\linewidth}
\centering
 \begin{tikzpicture}[
        V/.style = {
                    circle,thick,fill=white},scale=0.5]
        
       \node[V] (1) at (1,5) [scale=0.6,draw,fill=black!15] {1};
        \node[V] (2) at (11,7) [scale=0.6,draw,fill=black!15] {2};
        \node[V] (3) at (11,5) [scale=0.6,draw,fill=black!15] {3};
        \node[V] (4) at (11,3) [scale=0.6,draw] {4};
        \node[V] (5) at (9,4) [scale=0.6,draw] {5};
        \node[V] (6) at (3,6) [scale=0.6,draw] {6};
        \node[V] (7) at (6,6) [scale=0.6,draw,fill=black!15] {7};
        \node[V] (8) at (9,2) [scale=0.6,draw,fill=black!15] {8};
        \node[V] (9) at (6,3) [scale=0.6,draw,fill=black!15] {9};
        \node[V] (10) at (3,3) [scale=0.6,draw] {10};

        \draw[black,thick]
        (1) to (6)
        (1) to (10)
        (6) to (7)
        (6) to (10)
        (7) to (5)
        (10) to (9)
        (9) to (5)
        (5) to (2)
        (5) to (3)
        (5) to (4)
        (3) to (4)
        (4) to (8);
    \end{tikzpicture}
\caption{}
\label{fig:fig15}
\end{minipage}
\end{figure}   
\end{example}


\begin{remark}
    If for a maximal independent set $A$ we tweak the definition of the external activity as follows
    $$Ext(A)=\{v\in A^{\complement}: \exists a\in A:a\in N(v) \text{ with } v<a\}.$$
    Then for every graph it is possible to construct a unique complete maximal independent set which generates all other vertex subsets. The proof would be basically a mixture of arguments from Theorem \ref{extcompleteth} and Theorem \ref{intcompletecons}.
\end{remark}

\begin{theorem}\label{th4}
   Let  $G$ be a graph with vertex set $V=\{1,\ldots,n\}$, then the following hold:
   \begin{enumerate}
       \item If $G$ has a complete maximal independent set then its cover does not generate a partition.
       \item A graph with two internally complete sets cannot generate a partition.
   \end{enumerate}
\end{theorem}

\begin{proof}
\leavevmode 
\begin{enumerate}
    \item \label{point1} The complete maximal independent set $A$ of $G$ generates the interval 
     \begin{equation}
        [A\setminus Int(A);A\cup Ext(A)]=[\emptyset,V]=2^V.
    \end{equation}
     For any maximal independent set $B$ we have 
     $$[B\setminus Int(B);B\cup Ext(B)]\cap [\emptyset,V] \neq \emptyset. $$
     Since the intervals are not disjoint the cover \eqref{cover} does not generate a partition.
     \item  Let $A$ and $B$ be two internally complete sets of $G$. Then they generate respectively the intervals
     $$[\emptyset;A\cup Ext(A)] \text{ and } [\emptyset; B\cup Ext(B)].$$
     The empty set is generated twice and therefore the cover is not a partition.
\end{enumerate}
\end{proof}

\begin{definition}
    A maximal independent set $A$ is called internally empty if $Int(A)=\emptyset$.
\end{definition}

\begin{definition}
    A maximal independent set $A$ is called externally empty if $Ext(A)=\emptyset$.
\end{definition}

\begin{proposition}
    Let $G=(V,E)$ be a graph with family of maximal independent sets $\mathscr{M}$. If for some $A\in \mathscr{M}$, $Ext(A)=\emptyset$ then $Int(A)=A$.
\end{proposition}

\begin{proof}
    Assume the opposite, that there exists  $a\in A$ which is not an internally active vertex. This means that there exists $b\in N(a)$ with $b>a$ such that $(A\setminus \{a\})\cup \{b\}$ is independent. As a consequence, by definition of external activity, $b$ is an external active vertex with respect to $A$, a contradiction. 
\end{proof}



\begin{theorem}
   Let $G$ be a graph with the family of maximal independent sets $\mathscr{M}(G)=\mathscr{M}$. For any $v\in V$, $\exists   A\in \mathscr{M}$ with $v\in A$ such that either $A\setminus Int(A)=\{v\}$ or $Int(A)=A$. 
\end{theorem}

\begin{proof}
    Apply Algorithm \ref{Alg1} to $G\big[V\setminus N[v]\big]$. The set $A=S\cup \{v\}$ is a maximal independent set. Since the set $S$ is internally complete we need to check only if the vertex $v$ is internally active in $A$. There are two cases:
    \begin{enumerate}
        \item $\forall u\in N(v):u<v$. Then $v$ cannot be substituted by a greater neighbour in $A$ since all of them are smaller. Therefore $v$ is internally active in $A$ and $Int(A)=A$.
        \item $\exists u\in N(v):u>v$. Let $B=\{u\in N(v):u>v\}$, that is, the neighbours of $v$ which are greater than $v$. There are two subcases:
        \begin{enumerate}
            \item Every element in $B$ is adjacent to some vertex in $S$. Therefore $v$ can not be substituted by any of its greater neighbours in $A$ while maintaining independence. Thus $Int(A)=A$.
            \item There exists an element $b\in B$ which is not adjacent to any vertex in $S$. This means that $(A\setminus \{v\})\cup \{b\}$ is independent. Therefore according to the definition of internal activity $v$ is the only internally active vertex in $A$ and $A\setminus Int(A)=\{v\}$. 
        \end{enumerate}
    \end{enumerate}
\end{proof}

\begin{proposition}
   Let $G$ be a graph and $\mathscr{M}$ its family of maximal independent sets. If $G$ has a vertex $v$ such that $G\big[V\setminus N[v] \big]$ has isolated vertices then all sets $A\in \mathscr{M}$ which contain $v$ satisfy $Int(A)\neq \emptyset$.
\end{proposition}
\begin{proof}
    Assume the opposite, that there exists $ A\in \mathscr{M}$ which contains $v$ such that $Int(A)=\emptyset$ and let $B$ be the set of isolated vertices in $G\big[V\setminus N[v]\big]$. Since the vertices of $B$ are adjacent only with vertices in $N(v)$ it follows that $B\subsetneq A$.  Moreover, since all the neighbours of $b\in B$ are also neighbours of $v$ we have
    \begin{equation}
        \big(A\setminus\{b\}\big)\cup \{u\} \text{ is not independent }, \; \forall u\in N(b). 
    \end{equation}
    This means that none of the vertices from $B$ can be substituted in $A$ while maintaining independence, therefore $Subs(b)=\emptyset$ for all $b\in B$ and $B\subseteq Int(A)$. As a consequence $Int(A)\neq \emptyset$, a contradiction to our initial assumption. 
\end{proof} 

\begin{proposition}
Let $G=(V,E)$ be a graph and $A,B$ two maximal independent sets of $G$. There exist unique sets $M$ and $N$ with $M\subseteq A, N\subseteq B$ such that $B=(A\setminus M)\cup N$. Moreover, $M\cap N=\emptyset$ and either $N\cap Ext(A)\neq \emptyset$ or $M \cap Ext(B)\neq \emptyset$. 
\end{proposition}
\begin{proof}
     Put $M=A\setminus B$ and $N=B\setminus A$. We notice immediately that the uniqueness of $M$ and $N$ is guaranteed by their definition. Also, $B=(A\setminus M)\cup N$ and $M\cap N=\emptyset$. Let $a$ be an element of $M$. Since $B$ is maximal independent and therefore dominating 
     \begin{equation}
        N(a)\cap N=\{a_1,\ldots,a_m\}\neq \emptyset, \quad m\leq |N|.
     \end{equation}
    If there exists $a_i$, $i\leq |N|$ such that $a_i>a$, then $a_i\in Ext(A)$. Since $a_i\in N$, it follows that $N\cap Ext(A)\neq \emptyset$. Otherwise, if $a>a_i$ for all $i\leq |N|$, then $a\in Ext(B)$ and therefore $M\cap Ext(B)\neq \emptyset$.
\end{proof}

\section{Special Cases}
In this section we will discuss classes of graphs for which we can generate an interval partition of the Boolean lattice of the vertex set with all maximal independent sets as generators.

\subsection{The Complete Graph}
For the complete graph $K_n$ the cover \eqref{cover} always  generates disjoint intervals.

\begin{theorem}\label{completegraph}
Let $K_n$ be the complete graph with vertices $V=\{1,2,\ldots,n\}$ then the cover constructed according to \eqref{cover} always generates a partition, i.e.,
$$2^V=\bigcupdot_{v \in V}[v\setminus Int(v); v\cup Ext(v)].$$
\end{theorem}

\begin{proof}
    Consider the set $\{i\}$ for $1\leq i\leq n-1$. The vertex $i$ is smaller than and adjacent to vertices $\{i+1,\ldots,n\}$ therefore it is not internally active and generates the interval
     \begin{equation*}
        [\{i\};\{i,i+1\ldots,n\}].
    \end{equation*}
    Now we consider the set $\{n\}$. The vertex $n$ is greater than and adjacent to all other vertices. Therefore it is internally active and no vertex is externally active with respect to this vertex. The generated interval is 
    \begin{equation}\label{setn}
        [\emptyset;\{n\}].
    \end{equation}
    Since intervals are disjoint the cover is a partition. This completes the proof.
\end{proof}

\subsection{The Join of the Complete Graph and the Empty Graph}

\begin{definition}[Graph Join]
    Let $G_1=(V_1,E_1)$ and $G_2=(V_2,E_2)$ be two graphs such that $V_1$ and $V_2$ are disjoint. The join $G_1+G_2$ of the graphs $G_1$ and $G_2$ is the graph $(V_1\cup V_2,E_1\cup E_2)$ together with all the edges formed by connecting all the vertices of $G_1$ with the vertices of $G_2$.
\end{definition}

\begin{theorem}\label{Kn+pending}
    Let $K_n$ be a complete graph on $n$ vertices and $E_m$ an empty graph on $m$ vertices where $V(K_n)=\{1,\ldots,n\}$ and $V(E_m)=\{n+1,\ldots,n+m\}$. The cover of $K_n+E_m$ always generates a partition.
\end{theorem}

\begin{proof}

\begin{figure}[H]
\centering

 \begin{tikzpicture}[
        V/.style = {circle},M/.style = {circle,inner sep=0pt,minimum size=0pt},scale=1.6]
        
        label={[shift={(1.0,0.3)}]Label}
        
    \draw[black,fill=black!15!] (4.49,5.68) circle (0.85cm);
    \node[M] (1) at (4.35,6.54)[label=above:$1$] {};
    \node[V] (kn) at (4.49,5.68)[fill=black!15!] {$K_n$};
    \node[M] (n) at (4.35,4.82) [label=below:$n$] {};
    \node[V] (n+1) at (7.5,6.54) [scale=0.55,draw,fill=black,label=right:$n+1$] {};
    \node[V] (n+2) at (7.5,5.95) [scale=0.555,draw,fill=black,label=right:$n+2$] {};
    \node[V] (n+m) at (7.5,5) [scale=0.55,draw,fill=black,label=right:$n+m$] {};
        
    \draw[black,thick,dotted]
        (n+1) to (1)
        (n+1) to (n)
        (n+2) to (1)
        (n+2) to (n)
        (n+m) to (1)
        (n+m) to (n);
        \path (n+2) -- (n+m) node [midway, sloped] {$\dots$};
    \end{tikzpicture}
\caption{}
\label{fig6}
\end{figure}

    Let $V_1=V(K_n)=\{1,\ldots,n\}$ and $V_2=V(E_m)=\{n+1,\ldots,n+m\}$.
    Note that by virtue of $G=K_n+E_m$ the graph $G[V_1\cup \{v\}]$ is complete for all $v\in V_2$. There are two kind of maximal independent sets in $K_n+E_m$:
    \begin{enumerate}
        \item The sets $\{1\},\{2\},\ldots,\{n\}$ which generate respectively the intervals 
        \begin{equation*}
            [\{1\};\; \{1,2,\ldots,n+m\}],\ldots,[\{n\};\;\{n,n+1,\ldots,n+m\}].
        \end{equation*}
        
        \item The set $V_2$. Since all the elements of this set are greater than all their neighbours in $K_n+E_m$ then $Int(V_2)=V_2$, $Ext(V_2)=\emptyset$ and the generated interval is $[\emptyset;\;V_2]$.
    \end{enumerate}
    Since the generated intervals are disjoint this ends the proof.
\end{proof}

\begin{remark}
    Theorem \ref{Kn+pending} provides a labelling algorithm which generates a partition for graphs which are join of a complete graph $K_n$ and an empty graph $E_m$: label the vertices of $K_n$ as $1,2,\ldots,n$ and the vertices of $E_m$ as $n+1,\ldots,n+m$.
\end{remark}

\begin{theorem}
   Let $K_n$  be the complete graph on $n$ vertices and let  $S_i$ be the set of vertices which are adjacent only to the vertex $i$ such that for all $v\in S_i$ it holds $v>n$, $i=1,\ldots,n$ (Figure \ref{fig14}). The cover of the graph formed by $K_n$ together with its pendant vertices generates a partition if and only if 
   \begin{enumerate}
       \item \label{11} $|S_i|>0$ for $i=1,\ldots,n$ or
       \item \label{22}If there exists some set $B\subseteq V(K_n)$ such that $|S_i|=0$ for $i\in B$ then $n\in B$.
   \end{enumerate}
    
\end{theorem}

\begin{proof}

\begin{figure}[H]
\centering
\begin{tikzpicture}[
        V/.style = {
                    circle,thick,fill=white},scale=1.5]
        
        label={[shift={(1.0,0.3)}]Label}
         \draw[black,fill=black!15!] (4,4) circle (1cm);
         \node (kn) at (4,4)[fill=black!15!] {$K_n$};
        \node[V] (i) at (4.5,4.25) [scale=0.3,draw,fill=black,label=above:$i$] {i};
        \node[V] (2) at (6,3.5) [scale=0.3,draw,fill=black] {2};
        \node[V] (3) at (6,4) [scale=0.3,draw,fill=black] {3};
        \node[V] (4) at (6,4.5) [scale=0.3,draw,fill=black] {4};

        \draw[black,thick]
        (i) to (2)
        (i) to (3)
        (i) to (4);
        \path (3) -- (2) node [midway, sloped] {$\dots$};
        \node[draw,ellipse,fit=(2)(3)(4),label=right:$S_i$]{};
\end{tikzpicture}
\caption{}
\label{fig14}
\end{figure}

$ $\newline
($\Longleftarrow$)
Let  $V=V(K_n), S=\bigcup\limits_{i=1}^n S_i$ and $A_i=\bigcup_{\substack{j=1 \\ j\neq i}}^n S_j$ for $1\leq i\leq n.$ We prove first that  Conditions \ref{11} and \ref{22} are sufficient for the cover to generate a partition.
    \begin{enumerate}
    \item
         For the first case there are two kinds of maximal independent sets: 
        
        \begin{enumerate}
            \item The sets $A_1\cup\{1\},\ldots,A_n\cup\{n\}$ which generate respectively the intervals
            $$[\{1\};\;V\cup S ],\ldots, [\{n\};\; (V\cup S)\setminus \{1,\ldots,n-1\}].$$
            
            \item The set $S$ which generates the interval $[\emptyset;\; S]$.
        \end{enumerate}
        The intervals are disjoint therefore the cover generates a partition.

    \item
    Let $C=V\setminus B$ . We have that $S_i$ are pendant vertices of $i\in C$ and $|S_{i}|>0$. For the second case there are two kinds of maximal independent sets:
        \begin{enumerate}
            \item \label{2a} The sets $A_i\cup\{i\}$ for $i\in C$ which generate the intervals $[\{i\};\;(V\cup S)\setminus \{1,\ldots,i-1\}].$
            \item \label{2b} The sets $S\cup \{i\}$ for $i\in B$ which generate the intervals $[\{i\};\;(V\cup S)\setminus \{1,\ldots,i-1\}]$
            when $i\neq \max \{B\}=n$. When $i=\max \{B\}=n$ the maximal independent set $S\cup \{i\}$ has internal activity $\emptyset$ since none of its elements can be substituted by any of its neighbours while maintaining independence. So it generates the interval 
            $$[\emptyset;\; (V\cup S)\setminus \{1,\ldots,\max B-1\}].$$
        \end{enumerate}
        The intervals are disjoint, therefore the cover generates a partition.
    \end{enumerate}
$ $\newline
($\Longrightarrow$) Assume that the graph $K_n$ together with pending vertices $S$ has a cover which generates a partition and that the Conditions \ref{11} and \ref{22} of the theorem do not hold. This means that there exists a set $B$ such that for all $i\in B$ we have $|S_i|=0$ but $n\not\in B$, that is, $|S_n|>0$. So the vertex $n$ has pending vertices $S_n$. Let $m=\max \{B\}$, the maximal independent set $S\cup \{m\}$ generates the interval 
    \begin{equation}
        [\emptyset;\; (V\cup S)\setminus \{1,\ldots,m-1\}]
    \end{equation}
which contains the set $\{n\}$ since $m<n$. The maximal independent set $A_n\cup \{n\}$ generates the interval 
    \begin{equation}\label{n-sets}
        [\{n\};\; S\cup \{n\}].
    \end{equation}
which also contains  the set $\{n\}$, a contradiction to the fact that the cover is a partition.
\end{proof}

\subsection{The Lex and Colex Graphs}
\begin{definition}\cite{cutler2011extremal}
    Given $A,B\subset \mathbf{N}$ we say $A$ precedes $B$ in lexicographic (or lex) ordering, written $A{<}_L B$, if $\min\{A\triangle B\}\in A$.
\end{definition}

The lex graph, denoted $L(n,m)$, is the graph with vertex set $[n]=\{1,\ldots,n\}$ and edge set the first $m$ elements of $\binom{[n]}{2}$ under the lex ordering. The first few elements of the lex order on $\binom{[n]}{2}$ are

\begin{equation*}
\{1, 2\} , \{1, 3\} , \{1, 4\} ,\ldots, \{1, n\} , \{2, 3\} , \{2, 4\} ,\ldots , \{2, n\} , \{3, 4\} ,\ldots
\end{equation*}

 \begin{figure}[H]
    \centering
     \begin{tikzpicture}[
            V/.style = {
                        circle,thick,fill=white},scale=0.4]
            \node[V] (1) at (6,4) [scale=0.5,draw] {1};
            \node[V] (2) at (1,4) [scale=0.5,draw] {2};
            \node[V] (3) at (3,2) [scale=0.5,draw] {3};
            \node[V] (4) at (3,6) [scale=0.5,draw] {4};
            \node[V] (5) at (9,4) [scale=0.5,draw] {5};
            \draw[black,very thick]
            (1) to (2)
            (1) to (3)
            (1) to (4)
            (1) to (5)
            (2) to (3)
            (2) to (4);
    \end{tikzpicture}
    \caption{$L(5,6)$}
    \label{L(5,6)}
    \end{figure}
    
The class of the lex graphs is interesting because they are extremal for the number of the independent sets. Among the graphs with $n$ vertices and $m$ edges the lex graph $L(n,m)$  maximizes both, the count of total independent sets and the count of independent sets of fixed size \cite{cutler2011extremal}.
\begin{theorem}\label{exp1lemma}
    Let $m,n$ be two natural numbers such that $n-1\leq m\leq \binom{n}{2} $. Then there exists a unique sequence ${\{p_i\}_{i=1}^k}$ such that
\begin{equation}\label{expansion1}
    \sum \limits_{i=1}^k p_i=m
\end{equation}
where:
    \begin{enumerate}
        \item $p_i>0$ and $1 \leq k \leq n-1$,
        \item $p_i=n-i$ for $i\in\{1,\ldots,k-1\}$,
        \item $p_k\in\{1,\ldots,n-k\}$.
    \end{enumerate}
\end{theorem}
\begin{proof}
The proof can be done by double induction on $m,n$.
\end{proof}

\begin{definition}
    Let $m,n$ be two natural numbers such that $n-1\leq m\leq \binom{n}{2}$. The expansion \eqref{expansion1} satisfying conditions (1)-(3) in Theorem \ref{exp1lemma} is called subsequent decreasing summation decomposition of $m$ with base $n$ or shortly $sds(m,n)$. The number $k$ is called depth of the summation \eqref{expansion1} or the depth of $sds(m,n)$.
\end{definition}

Let $L(n,m)$ be a lex graph on $n$ vertices and $m$ edges such that $m\geq n$. The depth of $sds(m,n)$ is called the depth of the lex graph $L(n,m)$. By Theorem \ref{exp1lemma} such Lex graphs have depth at least $2$. For $1\leq i\leq n$ let $A(i)=\{v\in [n]: v<i\}$ and $B=\{k+1,\ldots,k+p_k\}$ then by virtue of the definition of $L(n,m)$ we have 

$$
N(i)=
\begin{cases}
V\setminus \{i\} &\text{ for } i\in\{1,\ldots,k-1\} ,\\
A(i)\cup B &\text{ for } i=k,\\
A(i) &\text{ for } i\in\{k+1,\ldots,n\}.
\end{cases}
$$

\begin{theorem}
    Let $L(n,m)$ be an arbitrary lex graph. The cover \eqref{cover} generates a partition.
\end{theorem}
\begin{proof}
    There are two cases:
    \begin{enumerate}
    \item
    $m< n-1$. Then the lex graph $L(n,m)$ is composed of the star graph $S_m$ where $1$ is central and the set of isolated vertices $S=\{m+1,\ldots,n\}$. There are two maximal independent sets: the set $\{1\}\cup S$ which generates the interval $[\{1\};\{1,\ldots,n\}]$ and the set $\{2,\ldots,n\}$ which generates the interval $[\emptyset;\{2,\ldots,n\} ]$.
    \item 
    $m\geq n-1$. Let $k$ be the depth of the lex graph $L(n,m)$. There are two cases:
    \begin{enumerate}
        \item \label{a} $p_k\neq n-k$. Then there are three types of maximal independent sets:
        \begin{itemize}
            \item
            The sets $\{i\}$ for $i\in \{1,\ldots,k-1\}$ which generate the intervals $[\{i\};\{i,i+1,\ldots,n\}]$, respectively.
            \item The set $\{k,k+p_k+1,\ldots,n\}$ which generates the interval $[\{k\};\{k,\ldots,n\}]$.
            \item The set $\{k+1,\ldots,n\}$ which generates the interval $[\emptyset; \{k+1,\ldots,n\}]$.
            \end{itemize}
        
        \item $p_k=n-k$. Then the maximal independent sets there are of the first and the third type presented in Case \ref{a}. 
    \end{enumerate}

    \end{enumerate}
    The intervals are disjoint. This ends the proof.
\end{proof}

\begin{example}
We will construct the partition cover of the lex graph $L(5,6)$ in Figure \ref{L(5,6)} which has depth $2$ and $p_2=2$.

\begin{align*}
    \{1\}&: [\{1\}; \{1,2,3,4,5\}\\
    \{2,5\}&: [\{2\}; \{2,3,4,5\}]\\
    \{3,4,5\}&: [\emptyset; \{3,4,5\}].
\end{align*}

\end{example}

\begin{definition}\cite{cutler2011extremal}
    Given $A,B\subset \mathbf{N}$ we say $A$ precedes $B$ in colexicographic (or colex) ordering, written $A{<}_C B$, if $\max\{A\triangle B\}\in B$.
\end{definition}

The colex graph, denoted $C(n,m)$, is the graph with vertex set $[n]$ and edge set the first $m$ elements of $\binom{[n]}{2}$ under the colex ordering. The first few elements of the colex order on $\binom{[n]}{2}$ are
\begin{equation*}
    \{1,2\}, \{1,3\}, \{2,3\}, \{1,4\}, \{2,4\},\{3,4\},\{1,5\},\ldots
\end{equation*}

\begin{figure}[H]
\centering
 \begin{tikzpicture}[
        V/.style = {
                    circle,thick,fill=white},scale=0.4]
        \node[V] (1) at (5,3) [scale=0.5,draw] {1};
        \node[V] (2) at (0,3) [scale=0.5,draw] {2};
        \node[V] (3) at (2,5) [scale=0.5,draw] {3};
        \node[V] (4) at (2,1) [scale=0.5,draw] {4};
        \node[V] (5) at (8,3) [scale=0.5,draw] {5};
        \node[V] (6) at (6,5) [scale=0.5,draw] {6};
        \draw[black,very thick]
        (1) to (2)
        (1) to (3)
        (1) to (4)
        (1) to (5)
        (2) to (3)
        (2) to (4)
        (3) to (4);
\end{tikzpicture}
\caption{$C(6,7)$}
\label{C(6,7)}
\end{figure}

The colex graph $C(n,m)$ has the fewest independent sets among threshold graphs on a given number of vertices and edges \cite{keough2018extremal}.
\begin{theorem}\label{exp2lemma}
    Let $m,n$ be two natural numbers such that $1\leq m \leq \binom{n}{2}$.  There exists a unique sequence ${\{q_i\}_{i=1}^k}$ such that
    
    \begin{equation}\label{expansion2}
    \sum \limits_{i=1}^k q_i=m
    \end{equation}
    where:
    \begin{enumerate}
        \item $q_i>0$ and $1 \leq k \leq n-1$,
        \item $q_i=i$ for $i\in\{1,\ldots,k-1\}$,
        \item $q_k\in\{1,\ldots,k\}$.
    \end{enumerate}
\end{theorem}
\begin{proof}
The proof can be done by double induction on $m,n$.
\end{proof}

\begin{definition}
    Let $m,n$ be two natural numbers such that $1\leq m \leq \binom{n}{2}$. The expansion \eqref{expansion2} satisfying conditions (1)-(3) in Theorem \ref{exp2lemma} is called subsequent increasing summation decomposition of $m$ with base $n$ or shortly $sis(m,n)$. The number $k$ is called depth of the summation \eqref{expansion2} or the depth of $sis(m,n)$.
\end{definition}

Let $C(n,m)$ be a colex graph on $n$ vertices and $m$ edges. The depth of $sis(m,n)$ is called the depth of the colex graph $C(n,m)$. By virtue of the definition of $C(n,m)$ we have:
$$
N(i)=
\begin{cases}
[k+1]\setminus \{i\} &\text{ for } i\in\{1,\ldots,q_k\},\\
[k]\setminus \{i\} &\text{ for } i\in\{q_k+1,\ldots,k\},\\
\emptyset &\text{ for } i\in\{k+1,\ldots,n\}.
\end{cases}
$$

\begin{theorem}
    Let $C(n,m)$ be an arbitrary colex graph on $n$ vertices and $m$ edges with depth $k$. The cover \eqref{cover} generates a partition.
\end{theorem}
     
\begin{proof}
    Let $S=\{v\in [n]:v>k+1\}$, this is the set of the isolated vertices of $C(n,m)$. We distinguish two cases:
    \begin{enumerate}
        \item \label{1} $q_k\neq k$. In this scenario there are three possible types of maximal independent sets:
    \begin{itemize}
        \item The sets $\{i\}\cup S$ for $i\in \{1,\ldots,q_k\}$ which generate the intervals $[\{i\}; \{i, \ldots,n\}]$.
        \item The sets $\{i,k+1\}\cup S$ for $i\in \{q_k+1,\ldots,k-1\}$ which generate the intervals \allowbreak $[\{i\}; \{i,\ldots,n\}]$.
        \item The set $\{k,k+1\}\cup S$  which generates the interval $[\emptyset,\{k,k+1\}\cup S]$.
    \end{itemize}
    \item $q_k=k$. There are two types of maximal independent sets:
    \begin{itemize}
        \item The sets $\{i\}\cup S$ for $i\in \{1,\ldots,k\}$ which generate the intervals $[\{i\}; \{i,\ldots,n\}]$.
        \item The set $\{k+1\}\cup S$ which generates the interval $[\emptyset; \{k+1\}\cup S]$.
    \end{itemize}
    \end{enumerate}
 The intervals are disjoint. This ends the proof.
\end{proof}

\begin{example}
We will construct the partition cover of the colex graph $C(6,7)$ in Figure \ref{C(6,7)} which has depth $4$ and $q_4=1$.

\begin{align*}
    \{1,6\}&: [\{1\}; \{1,2,3,4,5,6\}]\\
    \{2,5,6\}&: [\{2\}; \{2,3,4,5,6\}\\
    \{3,5,6\}&: [\{3\}; \{3,4,5,6\}]\\
    \{4,5,6\}&: [\emptyset; \{4,5,6\}].
\end{align*}    
\end{example}

\subsection{Pruned Graphs}
 In this section we will deal with rooted trees and graphs which are closely related them. Let $T$ be a tree which is rooted in the vertex $r$, where $deg(r)\geq 2$. The \textit{level} of a node $v$ of the  tree $T$, denoted $l(v)$, is the distance of the vertex $v$ from the vertex $r$ plus $1$, that is, $l(v)=d(v,r)+1$. The \textit{k-level sets} of a tree $T$, denoted $L_k(T)$ or $L_k$ when the tree $T$ is known from the context, are the sets $\{v\in V(T):l(v)=k\}$. As a consequence, the number of levels of $T$, denoted $l(T)$, can be written as 
    \begin{equation}
        l(T)=\max\{k: L_k \text{ is a k-level set of } T\}.
    \end{equation}
    
\begin{remark}
    From now and on we will assume that the trees are rooted in the center (or one of the centers, if there are two). The results hold also in the case when  the trees are rooted in a vertex $r$ such that $deg(r)\geq 2$. 
\end{remark}

Let $u,v$ be vertices of a tree $T$, the vertex $u$ is called \textit{child} of $v$ if $u\in N(v)$ and $l(u)>l(v)$. The set of children of $v$ is denoted $ch(v)$.  Similarly, for $S\subseteq V(T)$, the children of $S$, denoted $ch(S)$ is the union of children of its elements. That is,$$ch(S)=\bigcup_{v\in S}ch(v).$$ 

\begin{definition}
    Let $T$ be a tree on $n$ vertices. We say that $T$ is a pruned tree if every node of every level (except the last one) has at least one leaf child.
\end{definition}


\begin{definition}[Graph Union]
    Let $G_1=(V_1,E_1)$ and $G_2=(V_2,E_2)$ be two graphs. The union $G_1\cup G_2$ of the graphs $G_1$ and $G_2$ is the graph $(V_1\cup V_2,E_1\cup E_2)$.
\end{definition}

Let $L(T)$ be the set of leafs of the tree $T$ and let $v\in V(T)\setminus L(T)$. The graph $K(v)$ 
is the complete graph formed by $v$ and all non-leaf vertices which are on the same level as $v$. By definition, $K(v)=K(u)$ when $l(v)=l(u)$.
Let $T$ be a pruned tree and $H$ be a supergraph of $T$ and a subgraph of

\begin{equation}\label{prunedgraph}
    T\cup \bigcup\limits_{v\in V(T)\setminus L(T)}\Big( (T[v]+\bigcup\limits_{\substack{i\leq l(T)\\i:i-l(v)\geq 2}}L_i)\cup K(v)\Big).
\end{equation}

We call such a graph a \textit{pruned graph} of tree $T$. If we do not allow edges between the non-leaf vertices of the same level then we have an \textit{inter level pruned graph} (Figure \ref{fig:fig13}). All notions mentioned so far for pruned trees can be adapted for pruned graphs.

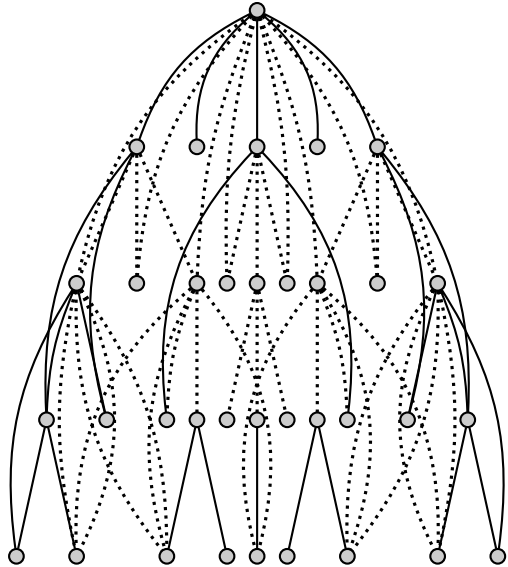
\begin{figure}[H]
\centering

 \begin{tikzpicture}[
        V/.style = {
                    circle,thick,fill=white},x=0.4cm,y=2ex]
    
        \node[V] (1) at (9,21) [scale=0.5,draw,fill=black!20] {};
        \node[V] (2) at (13,16) [scale=0.5,draw,fill=black!20] {};
        \node[V] (3) at (11,16) [scale=0.5,draw,fill=black!20] {};
        \node[V] (4) at (9,16) [scale=0.5,draw,fill=black!20] {};
        \node[V] (5) at (7,16) [scale=0.5,draw,fill=black!20] {};
        \node[V] (6) at (5,16) [scale=0.5,draw,fill=black!20] {};
        \node[V] (7) at (15,11) [scale=0.5,draw,fill=black!20] {};
        \node[V] (8) at (13,11) [scale=0.5,draw,fill=black!20] {};
        \node[V] (9) at (11,11) [scale=0.5,draw,fill=black!20] {};
        \node[V] (10) at (10,11) [scale=0.5,draw,fill=black!20] {};
        \node[V] (11) at (9,11) [scale=0.5,draw,fill=black!20] {};
        \node[V] (12) at (8,11) [scale=0.5,draw,fill=black!20] {};
        \node[V] (13) at (7,11) [scale=0.5,draw,fill=black!20] {};
        \node[V] (14) at (5,11) [scale=0.5,draw,fill=black!20] {};
        \node[V] (15) at (3,11) [scale=0.5,draw,fill=black!20] {};
        
        \node[V] (16) at (16,6) [scale=0.5,draw,fill=black!20] {};
        \node[V] (17) at (14,6) [scale=0.5,draw,fill=black!20] {};
        \node[V] (18) at (12,6) [scale=0.5,draw,fill=black!20] {};
        \node[V] (19) at (11,6) [scale=0.5,draw,fill=black!20] {};
        \node[V] (20) at (10,6) [scale=0.5,draw,fill=black!20] {};
        \node[V] (21) at (9,6) [scale=0.5,draw,fill=black!20] {};
        \node[V] (22) at (8,6) [scale=0.5,draw,fill=black!20] {};
        \node[V] (23) at (7,6) [scale=0.5,draw,fill=black!20] {};
        \node[V] (24) at (6,6) [scale=0.5,draw,fill=black!20] {};
        \node[V] (25) at (4,6) [scale=0.5,draw,fill=black!20] {};
        \node[V] (26) at (2,6) [scale=0.5,draw,fill=black!20] {};
        \node[V] (27) at (17,1) [scale=0.5,draw,fill=black!20!] {};
        \node[V] (28) at (15,1) [scale=0.5,draw,fill=black!20!] {};
        \node[V] (29) at (12,1) [scale=0.5,draw,fill=black!20!] {};
        \node[V] (30) at (10,1) [scale=0.5,draw,fill=black!20!] {};
        \node[V] (31) at (8,1) [scale=0.5,draw,fill=black!20!] {};
        \node[V] (32) at (6,1) [scale=0.5,draw,fill=black!20!] {};
        \node[V] (33) at (3,1) [scale=0.5,draw,fill=black!20!] {};
        \node[V] (34) at (1,1) [scale=0.5,draw,fill=black!20!] {};
        \node[V] (35) at (9,1) [scale=0.5,draw,fill=black!20!] {};

        \draw[dotted,very thick,black]
        (1) to[bend right=21] (15)
        (1) to[bend left=21] (7)
        (1) to[bend right=19] (14)
        (1) to[bend left=19] (8)
        (1) to[bend right=10]  (13)
        (1) to[bend left=10] (9)
        (1) to[bend right=8]  (12)
        (1) to[bend left=8] (10);
        
         \draw[thick,black]
        (6) to[bend right=21] (26)
        (2) to[bend left=21] (16)
        (6) to[bend right=21] (25)
        (2) to[bend left=21] (17)
        (4) to[bend right=25] (24)
        (4) to[bend left=25] (18);
        
        \draw[dotted,very thick,black]
        (13) to[bend right=22] (32)
        (13) to[bend right=26] (33)
        (15) to[bend right=12] (33)
        (15) to[bend right=20] (32)
        (7) to[bend left=20] (29)
        (9) to[bend left=26] (29)
        (9) to[bend left=26] (28)
        (7) to[bend left=12] (28)
        (13) to[bend left=27] (35)
        (9) to[bend right=27] (35)
        (15) to[bend left=27] (33)
        (7) to[bend right=27] (28)
        (15) to[bend left=18] (32)
        (7) to[bend right=20] (29)
        ;
        
        \draw[black,thick,black]
        (1) to[bend left=22] (2)
        (1) to[bend left=26] (3)
        (1) to (4)
        (1) to[bend right=26] (5)
        (1) to[bend right=22] (6);

        \draw[dotted,very thick,black]
        (6) to (13)
        (6) to (14)
        (6) to (15)
        (4) to (10)
        (4) to (11)
        (4) to (12)
        (2) to (7)
        (2) to (8)
        (2) to (9)
        (13) to[bend right=10] (24)
        (13) to (23)
        (9) to[bend left=10] (18)
        (9) to (19)
        (11) to (20)
        (11) to (21)
        (11) to (22);
        
         \draw[thick,black]
        (7) to[bend left=10] (16)
        (7) to (17)
        (15) to (25)
        (15) to[bend right=10] (26)
        (26) to (33)
        (26) to (34)
        (23) to (31)
        (23) to (32)
        (16) to (27)
        (16) to (28)
        (19) to (29)
        (19) to (30)
        (21) to (35)
        (15) to[bend right=20] (34)
        (7) to[bend left=20] (27);
    
       
        
    \end{tikzpicture}
\caption{An inter level pruned graph}
\label{fig:fig13}
\end{figure}

The graph in Figure \ref{fig:fig11} is a pruned tree and the graph in Figure \ref{fig:fig12} is one of its pruned graphs.
Let $T$ be a pruned tree with leaf set $L(T)$ and let $H$ be as in \eqref{prunedgraph}. The nodes belonging to the same level have the same color. The colored areas are cliques. By definition, a pruned graph can have a lot of cycles and cliques. Let the set of leafs of $H$ be $L(H)$. By \eqref{prunedgraph}, it is clear that $L(T)$ might be different from $L(H)$. By removing the leaves $L(H)$ from the graph $H$ we form the graph $H-L(H)$. Let the set of independent sets of $H-L(H)$ be $\mathscr{I}$ and the set of maximal independent sets of $H$ be $\mathscr{M}$. We can describe the maximal independent sets of $H$ by the independent sets of $H-L(H)$.

\begin{figure}[H]
\centering
\begin{minipage}[t]{.45\textwidth}
        \centering
 \begin{tikzpicture}[
        V/.style = {
                    circle,thick,fill=white},x=0.5cm,y=2.5ex]
    
        \node[V] (1) at (6,21) [scale=0.5,draw] {};
        \node[V] (2) at (6,19) [scale=0.5,draw,fill=orange] {};
        \node[V] (3) at (8,19) [scale=0.5,draw,fill=orange] {};
        \node[V] (4) at (4,19) [scale=0.5,draw,fill=orange] {};
        \node[V] (5) at (7,17) [scale=0.5,draw,fill=orange] {};
        \node[V] (6) at (5,17) [scale=0.5,draw,fill=orange] {};
        \node[V] (9) at (10,16) [scale=0.5,draw,fill=green] {};
        \node[V] (12) at (10,15) [scale=0.5,draw,fill=green] {};
        \node[V] (7) at (9,15) [scale=0.5,draw,fill=green] {};
        \node[V] (15) at (9,12) [scale=0.5,draw,fill=green] {};
        \node[V] (13) at (6,14.3) [scale=0.5,draw,fill=green] {};
        \node[V] (14) at (3,12) [scale=0.5,draw,fill=green] {};
        \node[V] (8) at (3,15) [scale=0.5,draw,fill=green] {};
        \node[V] (11) at (2,15) [scale=0.5,draw,fill=green] {};
        \node[V] (10) at (2,16) [scale=0.5,draw,fill=green] {};
        
        \node[V] (23) at (10.5,9) [scale=0.5,draw,fill=black!20] {};
        \node[V] (18) at (9,9) [scale=0.5,draw,fill=red] {};
        \node[V] (19) at (3,9) [scale=0.5,draw,fill=red] {};
        \node[V] (20) at (1.5,9) [scale=0.5,draw,fill=red] {};
        \node[V] (24) at (8,9) [scale=0.5,draw,fill=red] {};
        \node[V] (22) at (9.24,5.2) [scale=0.5,draw,fill=red] {};
        \node[V] (21) at (6,2.84) [scale=0.5,draw,fill=red] {};
        \node[V] (16) at (2.76,5.2) [scale=0.5,draw,fill=red] {};
        \node[V] (17) at (4,9) [scale=0.5,draw,fill=red] {};
        \node[V] (25) at (9.5,7) [scale=0.5,draw,fill=blue!40!] {};
        \node[V] (26) at (9,4) [scale=0.5,draw,fill=blue!40!] {};
        \node[V] (30) at (10,4) [scale=0.5,draw,fill=blue!40!] {};
        \node[V] (28) at (6.5,2) [scale=0.5,draw,fill=blue!40!] {};
        \node[V] (31) at (5.5,2) [scale=0.5,draw,fill=blue!40!] {};
        \node[V] (27) at (2.5,4) [scale=0.5,draw,fill=blue!40!] {};
        \node[V] (29) at (2,7) [scale=0.5,draw,fill=blue!40!] {};

        \draw[black,thick]
        (1) to (2)
        (1) to (3)
        (1) to (4)
        (1) to (5)
        (1) to (6);
        
        \draw[black,thick]
        (2) to (11)
        (2) to (12)
        (2) to (8)
        (2) to (13)
        (5) to (9)
        (5) to (14)
        (5) to (15)
        (6) to (7)
        (6) to (8)
        (6) to (10);
        
        \draw[black,thick]
        (7) to (23)
        (7) to (21)
        (15) to (16)
        (15) to (18)
        (15) to (17)
        (8) to (20)
        (8) to (24)
        (14) to (22)
        (14) to (19);
        
        \draw[black,thick]
        (24) to (25)
        (22) to (26)
        (22) to (30)
        (21) to (31)
        (21) to (28)
        (16) to (27)
        (17) to (29);
    \end{tikzpicture}
\caption{}
\label{fig:fig11}
\end{minipage}
  ~
\begin{minipage}[t]{.45\linewidth}
\centering
 \begin{tikzpicture}[
        V/.style = {
                    circle,thick,fill=white},x=0.5cm,y=2.5ex]
        
       \node[V] (1) at (6,21) [scale=0.5,draw] {};
        \node[V] (2) at (6,19) [scale=0.5,draw,fill=orange] {};
        \node[V] (3) at (8,19) [scale=0.5,draw,fill=orange] {};
        \node[V] (4) at (4,19) [scale=0.5,draw,fill=orange] {};
        \node[V] (5) at (7,17) [scale=0.5,draw,fill=orange] {};
        \node[V] (6) at (5,17) [scale=0.5,draw,fill=orange] {};
        \node[V] (9) at (10,16) [scale=0.5,draw,fill=green] {};
        \node[V] (12) at (10,15) [scale=0.5,draw,fill=green] {};
        \node[V] (7) at (9,15) [scale=0.5,draw,fill=green] {};
        \node[V] (15) at (9,12) [scale=0.5,draw,fill=green] {};
        \node[V] (13) at (6,14.3) [scale=0.5,draw,fill=green] {};
        \node[V] (14) at (3,12) [scale=0.5,draw,fill=green] {};
        \node[V] (8) at (3,15) [scale=0.5,draw,fill=green] {};
        \node[V] (11) at (2,15) [scale=0.5,draw,fill=green] {};
        \node[V] (10) at (2,16) [scale=0.5,draw,fill=green] {};
        
        \node[V] (23) at (10.5,9) [scale=0.5,draw,fill=red] {};
        \node[V] (18) at (9,9) [scale=0.5,draw,fill=red] {};
        \node[V] (19) at (3,9) [scale=0.5,draw,fill=red] {};
        \node[V] (20) at (1.5,9) [scale=0.5,draw,fill=red] {};
        \node[V] (24) at (8,9) [scale=0.5,draw,fill=red] {};
        \node[V] (22) at (9.24,5.2) [scale=0.5,draw,fill=red] {};
        \node[V] (21) at (6,2.84) [scale=0.5,draw,fill=red] {};
        \node[V] (16) at (2.76,5.2) [scale=0.5,draw,fill=red] {};
        \node[V] (17) at (4,9) [scale=0.5,draw,fill=red] {};
        \node[V] (25) at (9.5,7) [scale=0.5,draw,fill=blue!40!] {};
        \node[V] (26) at (9,4) [scale=0.5,draw,fill=blue!40!] {};
        \node[V] (30) at (10,4) [scale=0.5,draw,fill=blue!40!] {};
        \node[V] (28) at (6.5,2) [scale=0.5,draw,fill=blue!40!] {};
        \node[V] (31) at (5.5,2) [scale=0.5,draw,fill=blue!40!] {};
        \node[V] (27) at (2.5,4) [scale=0.5,draw,fill=blue!40!] {};
        \node[V] (29) at (2,7) [scale=0.5,draw,fill=blue!40!] {};
        
        \draw[fill=orange!25!,opacity=0.5]  (6,19) -- (7,17) -- (5,17);
        
        \draw[fill=green!15!,opacity=0.5]  (9,15) -- (9,12) -- (3,12)--(3,15);
        
        \draw[fill=red!15!,opacity=0.5]  (8,9) -- (9.24,5.2) -- (6,2.84)--(2.76,5.2)--(4,9);

        \draw[black,thick]
        (1) to (2)
        (1) to (3)
        (1) to (4)
        (1) to (5)
        (1) to (6);
        
        \draw[black,dotted,thick]
        (1) to  (8)
        (1) to (10)
        (1) to (7)
        (1) to (9)
        (1) to (12)
        (1) to (13)
        (1) to (14)
        (1) to (15)
        (1) to (11);

        \draw[black,dotted,thick]
        (2) to (5)
        (2) to (6)
        (5) to (6);
        
        \draw[black,thick]
        (2) to (11)
        (2) to (12)
        (2) to (8)
        (2) to (13)
        (5) to (9)
        (5) to (14)
        (5) to (15)
        (6) to (7)
        (6) to (8)
        (6) to (10) ;
        
        \draw[black,dotted,thick]
        (7) to (8)
        (7) to (14)
        (7) to (15)
        (8) to (14)
        (8) to (15)
        (14) to (15);
        
        \draw[black,thick]
        (7) to (23)
        (7) to (21)
        (15) to (16)
        (15) to (18)
        (15) to (17)
        (8) to (20)
        (8) to (24)
        (14) to (22)
        (14) to (19);
        
        \draw[black,dotted,thick]
        (13) to (28)
        (15) to (27)
        (14) to (26)
        (15) [bend left=15] to  (25)
        (14)[bend right=15]to (29)
        (5) to (23);
        
        \draw[black,dotted,thick]
        (16) to (17)
        (16) to (21)
        (16) to (22)
        (16) to (24)
        (17) to (21)
        (17) to (22)
        (17) to (24)
        (21) to (22)
        (21) to (24)
        (22) to (24);

        \draw[black,thick]
        (24) to (25)
        (22) to (26)
        (22) to (30)
        (21) to (31)
        (21) to (28)
        (16) to (27)
        (17) to (29);
    \end{tikzpicture}
\caption{}
\label{fig:fig12}
\end{minipage}

\end{figure}

In the following theorem, for notation convenience we will write $L$ instead of $L(H)$.

\begin{theorem}\label{th10}
Let $T$ be a pruned tree on $n$ vertices and $H$ defined as in \eqref{prunedgraph}. The function $f:\mathscr{I}\to \mathscr{M}$ with 

\begin{equation}\label{bijection}
    f(S)=S\cup \Big(  L\cap \big(V\setminus ch(S) \big) \Big)
\end{equation}

is a bijection.
\end{theorem}

\begin{proof}
     First we prove that the function $f$ is well defined. That is, if $S$ is an independent set of $H-L$ then $f(S)$ is a unique maximal independent set of $H$. The uniqueness of $f(S)$ is evident by the properties of the union operation. The set
     
     $$ L\cap \big(V\setminus ch(S) \big)$$
     
     is independent as a subset of leaves of $H$. On the other hand, none of its members is adjacent to any vertex in $S$. Therefore $f(S)$ is independent as union of two independent sets which have no edge in common. Now we will prove that $f(S)$ is a dominating set in $H$. Let $v\not\in f(S)$, by Eq. \eqref{bijection}, $v\not\in S$ and $v$ should be a child of  and therefore adjacent to some $u\in S\subseteq f(S)$.  Thus, $v\in N(u)\subseteq N(f(S))$.  Since $f(S)$ is also independent it is a maximal independent set of $H$.

    \textbf{Injectivity}. Let $S_1,S_2 \in \mathscr{I}$ such that $S_1\neq S_2$. Assume the opposite,  that $f(S_1)=f(S_2)$. This means that
    
    \begin{equation}
        S_1\cup \Big( L\cap (V\setminus ch(S_1) \big) \Big)=S_2 \cup \Big( L\cap (V\setminus ch(S_2) \big) \Big).
    \end{equation}
    Then we have 
    
    \begin{equation}\label{kot}
        S_1\setminus S_2
        =L\cap \Big( \big(V\setminus ch(S_1)\big) \setminus ch(S_2)\Big).
    \end{equation}
     Which is a contradiction because the LHS of Eq. \eqref{kot}  is a non-empty subsetset of non-leaf vertices of $T$ and the RHS is a subset of leaves of $T$.
     
    \textbf{Surjectivity}.
     Let $A\in \mathscr{M}$. If we remove from $A$ all the leaves then we are left with an independent set of vertices of $H-L$. 
\end{proof}

\begin{definition}
    Let $T$ be a pruned tree on $n$ vertices. We say that $T$ has a level labelling if for all $u,v\in V(T)$ it holds
    \begin{equation}
        l(u)<l(v) \Longrightarrow u<v .
    \end{equation}
\end{definition}

If the pruned tree of a pruned graph has level labelling then we say that a corresponding pruned graph has level labelling as well.

\begin{figure}[H]
    \centering
    \begin{minipage}[t]{.45\textwidth}
        \centering
         \begin{tikzpicture}[
        V/.style = {
                    circle,thick,fill=white},scale=0.455]
        
        \node[V] (1) at (6,8) [scale=0.555,draw] {1};
        \node[V] (6) at (0,6) [scale=0.555,draw] {6};
        \node[V] (4) at (3,6) [scale=0.555,draw] {4};
        \node[V] (5) at (6,6) [scale=0.555,draw] {5};
        \node[V] (2) at (9,6) [scale=0.555,draw] {2};
        \node[V] (3) at (12,6) [scale=0.555,draw] {3};
        \node[V] (11) at (1.5,4) [scale=0.5,draw] {11};
        \node[V] (9) at (4.5,4) [scale=0.555,draw] {9};
        \node[V] (10) at (9,4) [scale=0.5,draw] {10};
        \node[V] (8) at (12,4) [scale=0.555,draw] {8};
        \node[V] (7) at (15,4) [scale=0.555,draw] {7};
        \node[V] (12) at (4.5,2) [scale=0.5,draw] {12};
        \node[V] (13) at (10.5,2) [scale=0.555,draw] {13};
        \node[V] (14) at (13.5,2) [scale=0.5,draw] {14};

        \draw[black,very thick]
        (1) to (2)
        (1) to (3)
        (1) to (4)
        (1) to (5)
        (1) to (6)
        (4) to (11)
        (4) to (9)
        (3) to (10)
        (3) to (8)
        (3) to (7)
        (9) to (12)
        (8) to (13)
        (8) to (14);
    \end{tikzpicture}
    \caption{}
    \label{fig:fig8}
    \end{minipage}
    ~
\begin{minipage}[t]{.45\textwidth}
        \centering
         \begin{tikzpicture}[
        V/.style = {
                    circle,thick,fill=white},scale=0.455]
        
        \node[V] (1) at (6,8) [scale=0.555,draw] {1};
        \node[V] (6) at (0,6) [scale=0.555,draw] {6};
        \node[V] (4) at (3,6) [scale=0.555,draw] {4};
        \node[V] (5) at (6,6) [scale=0.555,draw] {5};
        \node[V] (2) at (9,6) [scale=0.555,draw] {2};
        \node[V] (3) at (12,6) [scale=0.555,draw] {3};
        \node[V] (11) at (1.5,4) [scale=0.5,draw] {11};
        \node[V] (9) at (4.5,4) [scale=0.555,draw] {9};
        \node[V] (10) at (9,4) [scale=0.5,draw] {10};
        \node[V] (8) at (12,4) [scale=0.555,draw] {8};
        \node[V] (7) at (15,4) [scale=0.555,draw] {7};
        \node[V] (12) at (4.5,2) [scale=0.5,draw] {12};
        \node[V] (13) at (10.5,2) [scale=0.555,draw] {13};
        \node[V] (14) at (13.5,2) [scale=0.5,draw] {14};

        \draw[black,very thick]
        (1) to (2)
        (1) to (3)
        (1) to (4)
        (1) to (5)
        (1) to (6)
        (4) to (11)
        (4) to (9)
        (3) to (10)
        (3) to (8)
        (3) to (7)
        (9) to (12)
        (8) to (13)
        (8) to (14);
        
        \draw[black,dotted,thick]
       (1) to (9)
       (1) to (10)
       (4) to (12)
       (3) to (13)
       (3) to (14)
       (1)[bend left=10] to (11)
       (4) to [out=-20,in=-160] (3)
       (9) to [out=-20,in=-160] (8)
       (1) to [bend right=20] (13);
    \end{tikzpicture}
    \caption{}
    \label{fig:fig9}
    \end{minipage}
\end{figure}
The graph in Figure \ref{fig:fig8} is a pruned tree $T$ with level labelling. The graph in Figure \ref{fig:fig9} is a pruned graph of the tree $T$ as defined in \eqref{prunedgraph}.
The following is a property which can be observed on pruned graphs with level labelling.

\begin{theorem}\label{property}
    Let $T$ be a pruned tree on $n$ vertices with level labelling, $H$ a pruned graph of $T$ and $A$ a maximal independent set of $H$. The only internally active vertices of $A$ are those which are leafs in $T$ and any such vertices do not contribute to $Ext(A)$.
    
\end{theorem}

\begin{proof}

   Let $v\in A$ be a leaf in $T$. Since the only neighbour (neighbors) of $v$ in $H$ has (have) smaller label then $Ext(v)=\emptyset$. Also, since $v$ is greater than all its neighbours in $H$ which can substitute it while maintaining independence then $v$ is internally active in $A$. No non-leaf vertex $u\in A$ can be internally active in $A$ because it can be substituted by its leaf child (children) with greater label. This ends the proof. 
\end{proof}

By combining Theorem \ref{th10} and Theorem \ref{property} we express the non internally active elements of a set $S$ by its pre-image.
\begin{theorem}
    Let $T$ be a pruned tree on $n$ vertices with level labelling, $f$ defined as in Theorem \ref{th10} and $H$ a pruned graph of $T$. Then a maximal independent set $A$ of $H$ generates the interval 
    \begin{equation}
        [f^{-1}(A);\; A\cup Ext(A)].
    \end{equation}
\end{theorem}

\begin{proof}
    In order to prove the theorem we will show that $A\setminus Int(A)=f^{-1}(A)$. By Theorem \ref{th10}, there exists a unique $S\in \mathscr{T}$ such that $f(S)=A$, which means
    \begin{equation}
        f(S)=S\cup \bigcup\limits_{v\in L\cap (V\setminus ch(S))}\{v\}=A.
    \end{equation}
    By Theorem \ref{property}, only leafs can be internally active. By construction, $S$ contains no leaf. This implies
    \begin{equation*}
        Int(f(S))= \bigcup\limits_{v\in L\cap (V\setminus ch(S))}\{v\}.
    \end{equation*}
    Therefore 
    \begin{align*}
        f(S)\setminus Int(f(S))=S=A\setminus Int(A).
    \end{align*}
    Since $f(S)=A$, then $S=f^{-1}(A)$ and $A\setminus Int(A)=f^{-1}(A)$.
\end{proof}

Now we will show that the level labelling always generates a partition on pruned graphs.

\begin{theorem}
    If $T=(V,E)$ is a pruned tree with level labelling and $H$ is a pruned graph of $T$. Then the set of intervals
    \begin{equation}\label{interval}
        \{[f^{-1}(A);A\cup Ext(A)], A\in \mathscr{M}(H)\},
    \end{equation}
    is a partition of $2^V$.
\end{theorem}

\begin{proof}
    We need to show that the intervals \eqref{interval} are disjoint for different sets. Assume the opposite, that there exist $B,C\in \mathscr{M}$ and a nonempty set $D$ such that 
    \begin{equation}
    D\in [f^{-1}(B);B\cup Ext(B)] \text{ and } D\in[f^{-1}(C);C\cup Ext(C)].
    \end{equation}
    This means that $M=f^{-1}(B)\cup f^{-1}(C)\subset D$. Let $N=f^{-1}(B)\cap f^{-1}(C)$ and let 
    \begin{equation}\label{elementa}
        a=\min \big\{M\setminus N\big\}=\min\big\{f^{-1}(B)\triangle f^{-1}(C)\big\},
    \end{equation}
    that is, the smallest among elements which are in either $f^{-1}(B)$ or $f^{-1}(C)$ and not in their intersection.
    Without loss of generality assume $a\in f^{-1}(B)$. Consequently, $a\not\in f^{-1}(C)$. Since $a\in D$ and $D$ is contained in the interval generated by the set $C$ then $a\in Ext(C)$. By Theorem \ref{property}, only non-internally active vertices of $C$ contribute to $Ext(C)$, therefore $a\in Ext(f^{-1}(C))$. Moreover, since $a\in f^{-1}(B)$ and $f^{-1}(B)$ is an independent set we have
    
    \begin{equation}
        a\in Ext(f^{-1}(C)\setminus N).
    \end{equation}
    Which means that $a$ is adjacent to and greater than an element of $f^{-1}(C)\setminus N\subseteq M\setminus N$, a contradiction to \eqref{elementa}.
\end{proof}

\begin{remark}
    It is evident that Theorems 7-10 hold also in the special case of pruned trees. 
\end{remark}

\begin{example}\label{ex5}
    We will construct the partition cover of the pruned graph with level labelling in Figure \ref{fig:fig9} where $V=\{1,2,\ldots,14\}$. 
    
\begin{align*}
        \{1,7,8,12\} &: [\{1,8\};\;V]&\\
        \{1,7,12,14\} &: [\{1\};\;V\setminus\{8\}]\\[+2mm]
        \{2,3,5,6,9,11\}&: [\{3,9\};\; V\setminus\{1\}]\\
        \{2,3,5,6,11,12\}&: [\{3\};\;V\setminus \{1,9\}]\\[+2mm]
        \{2,4,5,6,7,8,10\}&: [\{4,8\};\;V\setminus\{1,3\}]\\
        \{2,4,5,6,7,10,13,14\}&: [\{4\};\; V\setminus\{1,3,4\}]\\[+2mm]
        \{2,5,6,7,8,10,11,12\}&: [\{8\},V\setminus\{1,3,4\}]\\[+2mm]
        \{2,5,6,7,9,10,11,13,14\}&: [\{9\};\; V\setminus\{1,3,4,8\}]\\[+2mm]
        \{2,5,6,7,10,11,12,13,14\}&: [\{\emptyset\};\; V\setminus\{1,3,4,8,9\}].
\end{align*}
\end{example}

\section{Open Problems}
\begin{enumerate}
    \item Can we generate an interval partition of the independence complex $\mathscr{I}(G)=\{X\subseteq V: X \text{ is independent}\}$?
    \item Let $S$ be a maximal independent set and $X=S\cup Ext(S)$. Is there an independent set $S'\subseteq X$ with $|S'|>|S|$? Are there any conditions for the existence of $S'$?
    \item Can we apply the same idea to other set families, e.g. the neighborhood complex?
    \item Can we find a characterization of the pruned graphs that simplifies their recognition?
    \item Is the definition of a graph polynomial like
    
    $$f(G;x,y,z)=\sum\limits_{S\in \mathscr{M}(G)}x^{|S|}y^{|Ext(S)|}z^{|Int(S)|}$$
    useful?
\end{enumerate}
    \footnotesize
    \printbibliography

@article{crapo1969tutte,
  title={{The Tutte polynomial}},
  author={Crapo, Henry H},
  journal={{Aequationes Mathematicae}},
  volume={3},
  number={3},
  pages={211--229},
  year={1969},
  publisher={Springer}
}

@article{tutte1954contribution,
  title={{A contribution to the theory of chromatic polynomials}},
  author={Tutte, William Thomas},
  journal={{Canadian Journal of Mathematics}},
  volume={6},
  pages={80--91},
  year={1954}
}

@article{gordon1990generalized,
  title={{Generalized activities and the Tutte polynomial}},
  author={Gordon, Gary and Traldi, Lorenzo},
  journal={{Discrete Mathematics}},
  volume={85},
  number={2},
  pages={167--176},
  year={1990},
  publisher={Elsevier}
}

@article{bernardi2008characterization,
  title={{A characterization of the Tutte polynomial via combinatorial embeddings}},
  author={Bernardi, Olivier},
  journal={{Annals of Combinatorics}},
  volume={12},
  number={2},
  pages={139--153},
  year={2008},
  publisher={Springer}
}

@article{gessel1996tutte,
  title={{The Tutte polynomial of a graph, depth-first search, and simplicial complex partitions}},
  author={Gessel, Ira M and Sagan, Bruce E},
  journal={{The Electronic Journal of Combinatorics}},
  volume={3},
  number={2},
  note={{\#R9}},
  year={1996}
}

@book{las1984tutte,
  title={{The Tutte polynomial of a morphism of matroids II. Activities of orientations}},
  author={Las Vergnas, Michel},
  year={1984},
  publisher={Academic Press}
}

@article{brylawski1992tutte,
  title={{The Tutte polynomial and its applications}},
  author={Brylawski, Thomas and Oxley, James},
  journal={{Matroid Applications}},
  volume={40},
  pages={123--225},
  year={1992}
}

@incollection{ellis2011graph,
  title={{Graph polynomials and their applications I: The Tutte polynomial}},
  author={Ellis-Monaghan, Joanna A and Merino, Criel},
  booktitle={Structural analysis of complex networks},
  pages={219--255},
  year={2011},
  publisher={Springer}
}

@article{sokal2005multivariate,
  title={{The multivariate Tutte polynomial (alias Potts model) for graphs and matroids.}},
  author={Sokal, Alan D and others},
  journal={{Surveys in Combinatorics}},
  volume={327},
  pages={173--226},
  year={2005}
}

@article{AMC229,
  title={{From spanning forests to edge subsets}},
  author={Martin Trinks},
  journal={{Ars Mathematica Contemporanea}},
  volume={7},
  number={1},
  pages={141--151},
  year={2014},
}

@article{cutler2011extremal,
  title={Extremal problems for independent set enumeration},
  author={Cutler, Jonathan and Radcliffe, AJ},
  journal={{The Electronic Journal of Combinatorics}},
  volume={18},
  number={1},
  note={{\#P169}},
  year={2011}
}

@article{keough2018extremal,
  title={Extremal Threshold Graphs for Matchings and Independent Sets},
  author={Keough, Lauren and Radcliffe, AJ},
  journal={{Graphs and Combinatorics}},
  volume={34},
  number={6},
  pages={1519--1537},
  year={2018},
  publisher={Springer}
}
\end{document}